\documentclass[11pt]{amsart}
\usepackage{amsfonts}
\usepackage{stmaryrd}
\usepackage{mathrsfs}
\usepackage{amssymb,amsmath,amsthm,latexsym}
\usepackage{epsfig}
\usepackage{graphicx}
\usepackage{appendix}
\usepackage{lineno,version}

\font\tenbi=cmmib10   at 11 pt
\font\sevenbi=cmmib10 at 9pt
\font\fivebi=cmmib7 at 6pt
\newfam\bifam
\textfont\bifam=\tenbi \scriptfont\bifam=\sevenbi  \scriptscriptfont\bifam=\fivebi

\def\bi{\fam\bifam\tenbi}
\font\sixtdb=msbm10 at 16 pt \font\tendb=msbm10 at 12 pt  \font\sevendb=msbm7
\newfam\dbfam
\textfont\dbfam=\sixtdb

\textfont\dbfam=\tendb \scriptfont\dbfam=\sevendb

\newtheorem{thm}{Theorem}[section]
\newtheorem{lem}{Lemma}[section]

\newtheorem{example}{Example}[section]

\def\bq{\begin{equation}}
\def\eq{\end{equation}}
\def\ben{\begin{eqnarray*}}
\def\een{\end{eqnarray*}}
\def\br{\begin{eqnarray}}
\def\er{\end{eqnarray}}
\def\brr{\bq\begin{array}{rcl}}
\def\err{\end{array}\eq}
\def\bex{\begin{equation*}}
\def\eex{\end{equation*}}

\def\d{{\mathrm {d} }}
\def\E{\mathcal{S}}
\def\F{\mathcal{F}}
\def\C{\mathcal{C}}
\def\G{\mathcal{G}}
\def\P{\mathcal{P}}
\def\I{\mathcal{I}}

\def\text#1{\hbox{#1}}
\def\Dt{\triangle t}

\newcommand{\ds}{\displaystyle}
\newcommand{\nn}{\nonumber}
\newcommand{\bsub}{\begin{subequations}}
\newcommand{\esub}{\end{subequations}$\!$}

\begin{document}
\title[Parallel in time
for Volterra integral equations]
{Parallel in time algorithm
with spectral-subdomain enhancement
for Volterra integral equations}

\author{Xianjuan  Li}
\address{College of Mathematics and
Computer Science, Fuzhou University,  Fuzhou, 350108,  China}
\email{xjli\_math@yahoo.com.cn}

\author{Tao Tang}
\address{Department of Mathematics, Hong Kong Baptist University, Kowloon
Tong, Hong Kong}
\email{ttang@math.hkbu.edu.hk}
\author{Chuanju Xu}
\address{School of Mathematical Sciences,
Xiamen University, 361005 Xiamen, China.
The research of this author
was partially supported by National NSF of China (Grant numbers 11071203 and 91130002).
}
\email{cjxu@xmu.edu.cn}

\begin{abstract}
This paper proposes a parallel in time (called also time parareal) method
to solve Volterra integral equations of the second kind.
The parallel in time approach follows the same spirit 
as the domain decomposition that consists of
breaking the domain of computation into subdomains 
and solving iteratively the sub-problems
in a parallel way. 
To obtain high order of accuracy, a spectral collocation 
accuracy enhancement in subdomains will be employed.
Our main contributions in this work are two folds:
(i) a time parareal method is designed for the integral equations, 
which to our knowledge is the first of its kind. 
The new method is an iterative process combining a coarse prediction in the whole domain 
with fine corrections in subdomains by using spectral approximation,
leading to an algorithm of very high accuracy; (ii) 
a rigorous convergence analysis of the overall method is provided.
The numerical experiment confirms that the 
overall computational cost is considerably reduced
while the desired spectral rate of convergence can be obtained.

\end{abstract}


\date{\today}
\maketitle

\section{Introduction}
We consider the linear Volterra integral equations (VIEs) of the
second kind
\bq \label{1e1}
{u}(t)-
\int_0^tK(t,s){u}(s) \d s =g(t), \qquad \forall t\in I:=[0,T],
\eq
where $g$ and $K$ are sufficiently smooth in $I$ and
$I\times I$, respectively; and $K(t,t)\neq0$ for all $t\in I$.
Under these assumptions, smooth
solution $u(t)$ to \eqref{1e1} exists and unique, see, e.g.,
\cite{brunner2004collocation}.

The presence of the integral in (\ref{1e1})  makes the problem
{\em globally} time-dependent.
This means that the solution at time $t_k$ depends
on the solutions at all previous time $t<t_k$. Consequently, it may
require large storage if the solution time is large
or high accuracy of approximations is needed.
This may become more serious if partial integro-differential
equations are considered. To handle this, we will
design our method by breaking the domain into sub-domains
(as in the domain decomposition approach) and then solving
iteratively the sub-problems in parallel way.
More precisely, we divide the time interval
$[0, T]$ into $N$ equi-spaced subintervals
and then break the original problem into a series
of independent problems on the small sub-intervals.
These independent problems are solved by a fine approximation
which can be implemented in a parallel way, together with some coarse
grid approximations which have to be implemented in a sequential way.

The parallel in time algorithm for a model ordinary differential equation
(ODE) was initially introduced by Lions, Maday, and
Turinici \cite{lions2001} for solving evolution problems in
parallel.
It can be interpreted as a predictor-corrector
scheme \cite{bal2003parallelization, bal23parareal},
which involves a prediction step based on a coarse
approximation and a correction step computed in parallel
 based on a fine approximation.
Even though the time direction seems intrinsically
sequential, the combination of a coarse
and a fine solution procedure has proven to
allow for more rapid (convergent) solutions if
parallel architectures are available.
The parallel in time algorithm has received
considerable attention over the past years, especially
in the community of the domain decomposition
methods \cite{debit2002domain},
fluid and structure problems \cite{farhat2003time},
the Navier-Stokes equations \cite{fischer2005parareal},
quantum control problems \cite{maday2003parallel}, and so on.
For the parallel in time method based on the finite difference
scheme, the convergence analysis for an ODE
problem was given in \cite{gander29analysis}.

In our algorithm, we also build in a recent spectral method
approach to obtain exponential rate of convergence, see
\cite{chen2010convergence,tang2008spectral,JSCxlt12}.
The main advantage of using high order methods for integral
equations is their low storage requirement with
the desired precision;
this advantage makes high order methods attractive.
Among the high order methods spectral methods are known
very useful due to its exponential rate of convergence,
which is also demonstrated for solving VIEs \cite{stw11,FMClt12}.
However, a drawback of the spectral method is also well-known; i.e.,
the matrix associated with the spectral method is full
and the computational cost grows more quickly than that of a low order method.
Thus for long integration it is desirable to combine the spectral method with
domain decomposition techniques to avoid using a single
high-degree polynomials.

It is noted that there have been great interests in
studying the parabolic integro-differential equations, in
particular the study of discontinuous-Galerkin methods,
see, e.g., \cite{LTW98,MBMS11}.
These studies are very relevant to the present study and it will be
interesting to extend the present study to parabolic
integro-differential equations.
Another class of relevant problems is about
space-time fractional diffusion equation
which have been extensively studied
\cite{li2009space,li2010space}.
It is noted that the problems mentioned above also
contain a weakly singular kernel in the memory term, which
can provide extra difficulties in convergence analysis,
see, e.g., the recent work on the spectral methods for
VIEs with a weakly singular
kernel \cite{chen2010convergence,FMClt12}.

This paper is the first attempt to approximate solutions
of the VIEs by the parallel in time method. In addition, this article
will provide a full convergence analysis
for the proposed method. We will also demonstrate numerically
the efficiency and the convergence of the proposed method on
some sample problems.

This paper is organized as follows. In section 2,
we construct the
parallel in time method based on the spectral collocation scheme for the underlying
equation, and outline the main results. 
Some lemmas useful for the convergence analysis are provided in section 3.
The convergence analysis for the proposed method in $L^{\infty}$ under some assumptions is given
in section 4.
Numerical experiments are carried out in section 5.
In the final section we analyze the parallelism efficiency of the overall algorithm,
together with a description of some implementation details.

\section{Outline of the Parallem in Time Method and the
Main Results}
\setcounter{equation}{0}

The time interval $I=[0,T]$ is first partitioned into $N$
subintervals,  determined by the grid points
$0=t_0<t_1<t_2< \cdots < t_N=T$
with  $t_n=n\Delta t, \Delta t={T}/{N}$.
We denote this partition by $I=\cup_{n=1}^N I_n$ with
$I_n=[t_{n-1},t_n]$.

\subsection{Outline of the method}

In the obvious way, the problem (\ref{1e1}) is equivalent to the following
system of $N$ integral equations:
\begin{equation} \label{n2e1}
u_n(t)-K_{nn} u_n(t) = g(t) + \sum ^{n-1}_{j=1} K_{nj} u_j(t), \quad
{\rm on} \;\; I_n=[t_{n-1}, t_n],\ 1\le n \le N.
\end{equation}
with $u_j =u |_{I_j}$,
\begin{equation} \label{n2e2}
K_{nn} u_n(t) := \int^t_{t_{n-1} } K(t,s) u_n(s)ds, \;\;\;
{\rm and} \;\;
K_{nj} u_j(t) := \int^{t_j}_{t_{j-1}} K(t,s) u_j(s)ds, \;\;
1\le j \le n-1.
\end{equation}
The solution of (\ref{n2e1}) can be expressed through the solution operators $\E_n$ as follows:
\begin{equation} \label{n2e3}
u_n(t)=\E_n (t; u_1, \cdots, u_{n-1}) \;\; {\rm for} \; t\in I_n.
\end{equation}
Likewise, at the discrete level, we have the solution operators $\F_n$
for the collocation method to be described later:
\begin{equation} \label{n2e4}
U_n(t)=\F_n (t; U_1, \cdots, U_{n-1}) \;\; {\rm for} \; t\in I_n.
\end{equation}
with $U_n$ a polynomial of degree $M$. The main idea of the proposed
parallel in time method is to take a second family of discrete
solution operators, $\G_n$, defined in the same way using polynomials
of a lower degree $\tilde{M} < M$, and to set up the iteration:
\begin{equation} \label{n2e5}
U^k_n(t)=\G_n(t; U^k_1, \cdots, U^k_{n-1}) +
\C_n(t; U^{k-1}_1, \cdots, U^{k-1}_{n-1}), \;\; k\ge 1,
\end{equation}
with the initial value
\[
U^0_n = \G_n(t; U^0_1, \cdots, U^0_{n-1}),
\]
where the correction term in (\ref{n2e5}) is defined by
\begin{equation} \label{n2e5a}
 \C_n(t; U^{k-1}_1, \cdots, U^{k-1}_{n-1})
:= \F_n (t; U^{k-1}_1, \cdots, U^{k-1}_{n-1}) -\G_n
(t; U^{k-1}_1, \cdots, U^{k-1}_{n-1}).
\end{equation}
Here, the key is that the $\C_n$ can be computed in parallel for
$1\le n\le N$. The term $\G_n$ in (\ref{n2e5}) must be computed
sequentially, but this is relatively cheap if $\tilde {M}$ is small
compared to $M$. Thus, provided the iteration converges rapidly, we
can use multiple processors to obtain a high accuracy solution
in a small multiple of the time needed to compute a low
accuracy (sequential) solution.

It is pointed out that the way of presenting the
parallel in time method  places it in the category of
the predictor-corrector scheme,
where the predictor is $\G_{n} (t,U_{1}^k, \cdots, U_{n-1}^k) $
while the corrector is $\C_n(t; U^{k-1}_1, \cdots, U^{k-1}_{n-1})$.

\subsection{Outline of the main results}

The main result of this paper will be the following:
\begin{equation}
\Vert u_n - U^k_n\Vert _\infty
= O(M^{3/4-m} + (\tilde M/2)^{(3/4-m)(k+1)}),
\end{equation}
and numerical experiments confirm this convergence behavior
in practice. Thus $U^k_n$ is accurate to $O(M^{3/4-m})$ if
$M\approx (\tilde M/2)^{k+1}$.
We emphasize that the analysis covers a fully-discrete scheme in
which quadratures are used to approximate the integrals
that occur in the coefficients and right-hand side
of the resulting linear system that must be solved
at the $n$-th interval to compute $U^k_n$.

\subsection{The operators $\F_n$ and $\G_n$}

We begin by discussing the operator $\F_n$ defined in (\ref{n2e4}).
In this work,
we will use a {\em spectral collocation method}
to define this approximation.
Define $\mathcal{P}_{M}(I_n)$
as the polynomial spaces of degree less than or equal to $M$, with
$I_n=[t_{n-1},t_n]$.
Denote $L_{M}(x)$ the
Legendre polynomial of degree $M$.
Let $x_{i}$ be the points of the Legendre-Gauss (LG)
quadrature formula, defined by
$L_{M+1}(x_{i})=0, i=0, \cdots, M$, arranged by increasing order.
The associated weights of the LG quadrature formula are denoted
by
$\omega_{i}$, $0 \leq i \leq M$.
Then it is well-known the following identity:
\bex
\int_{-1}^1
\varphi(x)  \d x =
\sum^M_{i=0} \varphi(x_{i}) \omega_{i},
\;\;\;  \forall \varphi\in \mathcal{P}_{2M+1}(-1,1).
\eex
The discrete $L^2$ inner product associated to the LG quadrature is denoted by:
\br\label{dip}
(\varphi,\psi)_{M}
:= \sum^M_{i=0} \varphi(x_{i})\psi(x_{i}) \omega_{i}.
\er
Furthermore, define the LG points $\{\xi_n^i\}^M_{i=0}$ on the element
$I_n=[t_{n-1},t_n]$, i.e.,
\bex
\xi_n^i= \frac{t_n-t_{n-1}}{2} x_{i} + \frac{t_n+t_{n-1}}{2},
\quad  0\leq i\leq M, \ \  1\leq n\leq N,
\eex
and the corresponding weights
$\rho_n^i={\Delta t}\omega_{i}/2$.
Then it holds
\bex
\int_{I_n} \varphi(x)  \d x =
\sum^M_{i=0} \varphi(\xi_n^i) \rho_n^i,
\quad  \forall \varphi\in \mathcal{P}_{2M+1}(I_n),
\;\; 1\le n \le N.
\eex
We now use the Legendre-collocation method to
determine the operator $\F_n$ in (\ref{n2e4}). More precisely, we want to
find $U_n(t)\in \mathcal{P}_{M}(I_n)$ with $1\le n\le N$
such that
for all $0\leq i\leq M$,
\begin{equation}
\label{1e4}
U_n(\xi_n^i)-
\ds\int_{t_{n-1}}^{\xi_n^i}K(\xi_n^i,s) U_n(s) \d s
= g(\xi_n^i)+ \sum^{n-1}_{j=1} {\ds\int_{t_{j-1}}^{t_{j}}}
K(\xi_n^i,s) U_{j}(s) \d s,
\;\; 1\le n \le N.
\end{equation}
In the implementation, the integral terms on the left hand sides of \eqref{1e4} are evaluated by using
the following Gauss quadrature:
\bq \label{bK0}
 \int_{t_{n-1}}^{\xi_n^i}K(\xi_n^i,s) U_n(s) \d s
=\int_{-1}^{1}\bar K(\xi_n^i,s_n^i) U_n(s_n^i) \d x
\thickapprox
\left(\bar K(\xi_n^i,s_n^i), U_n(s_n^i)\right)_{M},
\eq
where
$
s_n^i:=s_n^i(x):= \frac{\xi_n^i-t_{n-1}}{2}x
+\frac{\xi_n^i+t_{n-1}}{2}, -1<x< 1
$, and
\bq\label{bK}
\bar K(\xi_n^i,s_n^i)
:=\left(\frac{\xi_n^i-t_{n-1}}{2}\right)K(\xi_n^i,s_n^i).
\eq
The right-hand sides of \eqref{1e4} are computed in a similar way:
\bq
\int_{t_{j-1}}^{t_j}K(\xi_n^i,s) U_j(s) \d s
= \frac{\Delta t}{2}\int_{-1}^{1}K(\xi_n^i,s_j) U_j(s_j) \d x
\thickapprox
\frac{\Delta t}{2}\left(K(\xi_n^i,s_j), U_j(s_j)\right)_{M},
\eq
where
\bq
\label{s_l}
s_j=s_{j}(x)= \frac{\Delta t}{2}x+\frac{t_j+t_{j-1}}{2},\quad -1<x<1.
\eq
With the help of the above numerical integrations,
\eqref{1e4} is further approximated by: $0\leq i\leq M$,
\bq
\label{1e5}
U_n(\xi_n^i)-
\left(\bar K(\xi_n^i,s_n^i), U_n(s_n^i)\right)_{ M}
= g(\xi_n^i)
+{\Dt\over 2}\sum^{n-1}_{j=1} \left(K(\xi_n^i,s_{j}),
U_{j}(s_{j})\right)_{M},
\;\; 1\le n \le N.
\eq
The above linear system defines the fine operator $\F_n$ in (\ref{n2e4}).

The operator $\G_n$ is defined by the same spectral collocation
method but with the degree $\tilde{M}$, which is much less than $M$.

Let
$\{\eta_n^i\}_{i=0}^{\tilde{M}}$ be a set of LG-points on the element  $[t_{n-1},t_n]$ corresponding
to the weight $\{\varrho_n^i\}_{i=0}^{\tilde{M}}$.
Then the coarse operator $\G_n(t, V_1, \cdots, V_{n-1})$ in \eqref{n2e5} consists in finding
$V_n(t)\in \P_{\tilde{M}}(I_n)$,
such that, for all $0\leq i\leq \tilde{M}$
\bq
\label{1e7}
 V_n(\eta_n^i)-
\left(\bar K(\eta_n^i,\tau_n^i), V_n(\tau_n^i)\right)_{ \tilde{M}}
= g(\eta_n^i)
+{\Delta t\over 2}\sum^{n-1}_{j=1}
\left(K(\eta_n^i,s_{j}), V_{j}(s_{j})
\right)_{\tilde {M}}, \quad
1\le n \le N,
\eq
where 
$
\tau_n^i:=\tau_n^i (x):= \frac{\eta_n^i-t_{n-1}}{2}x+\frac{\eta_n^i+t_{n-1}}{2},
\ -1<x<1,
$
and $\bar K$ and $s_j$ are respectively defined in \eqref{bK} and \eqref{s_l}.

\section{Useful Lemmas }
\setcounter{equation}{0}
We first introduce some notations.
Let $ \Lambda$ be an arbitrary bounded interval.
For non-negative integer $m$, 
$H^m(\Lambda)$ stands for the standard Sobolev space equipped with
the norm and seminorm $\|v\|_{m,\Lambda}$ and
$|v|_{k,\Lambda}, 0\le k \le m$. 
Particularly $L^2(\Lambda)=H^0(\Lambda)$, 
equipped with the standard $L^2$-inner product and norm.
Similarly, the norm of the space
$W^{m,\infty}(\Lambda)$ is denoted by
$\|v\|_{W^{m,\infty}(\Lambda)}$ or $\|v\|_{\infty,\Lambda}$ if $m=0$.
The error analysis needs  the following seminorm:
\bex
|v|_{H^{m;M}(\Lambda)}=\left(\sum_{k=\min(m,M+1)}^m
\|\partial_x^{(k)}v\|_{\Lambda}^2\right)^{\frac{1}{2}}.
\eex
Hereafter, in cases where
no confusion would arise, the domain symbol $\Lambda$  may be dropped from the notations.
We denote by $c$ generic positive constants independent
of the discretization parameters, but may depend on the kernel function $K(\cdot,\cdot)$ or
$T$.

We further introduce two approximation operators.
Firstly, we define the Lagrange interpolation operator
$\I_{\Lambda}^M: \
\mathcal{C}(\Lambda)\rightarrow\mathcal{P}_M(\Lambda)$,
by:
 $\forall v \in \mathcal{C}({\Lambda}), \I_\Lambda^M v\in \mathcal{P}_M(\Lambda)$,
 such that
\bex
\I_\Lambda^Mv(z_{i})=v(z_{i}), \quad 0\leq i\leq M,
\eex
where $\{z_{i}\}_{i=0}^M$ is the set of LG-points on the interval $\Lambda$.
This polynomial can be expressed as
\bex
\I_\Lambda^Mv(x)=\sum_{i=0}^Mv(z_i)h_\Lambda^i(x),
\eex
where $h_\Lambda^i$ is the Lagrange interpolation basis function associated with
$\{z_{i}\}_{i=0}^M$.
Particularly, we use $\I_n^M $ (resp., $h_n^i(x)$) to replace
$\I_\Lambda^M $ (resp., $h_\Lambda^i(x)$)
when $\Lambda=I_n$.

Then for all $v\in H^m(\Lambda), m\geqslant1$,
the following optimal error estimates hold (see, e.g.,
\cite{canuto2006spectral,stw11}):
\br
&& \|v-\I_\Lambda^M v\|
\lesssim
M^{-m}|v|_{H^{m; M}}, \label{3e3}
\\
 &&  \|v-\I_\Lambda^Mv\|_{\infty}
  \lesssim
  M^{\frac{3}{4}-m}|v|_{H^{m;M}}.
\label{3e5}
\er
For the discrete $L^2$-inner product defined in \eqref{dip},
it holds the following error estimate \cite{canuto2006spectral}:  \
$\forall \phi\in\mathcal{P}_M(\Lambda)$,
\begin{eqnarray}
|(v,\phi)-(v,\phi)_{M}|
\lesssim
M^{-m}|v|_{H^{m;M}}
\|\phi\|,\ \  v\in H^m(\Lambda), \ \  m\geq1.\label{3e4}
\end{eqnarray}

The following result gives
the Lebesgue constant for the Lagrange interpolation polynomials
associated with the LG-points.

\begin{lem} \label{lem3e4}
(p.329, eq.(9) in \cite{mastroianni2001optimal})
Let $\{h_\Lambda^j(x)\}_{j=0}^M$ be the Lagrange interpolation
basis associated with the $M+1$ LG-points on the interval $\Lambda$. Then
\bq \label{3e6}
\|I_\Lambda^M\|_{\infty}
:=\max_{x\in\Lambda}\sum_{j=0}^M|h_\Lambda^j(x)|
\lesssim \sqrt{M}.
\eq
\end{lem}

Using the standard Gronwall inequality, we obtain the following
estimate for the Volterra integral equations.

\begin{lem}\label{lem3}
Consider the following Volterra equation:
\bq\label{gron-eq1}
 u(t)=v(t)+\int_{0}^t K(t,s)u(s)\d s.
\eq
If $v\in W^{m,\infty}$ and $K(t,s)$ is sufficiently smooth, then
\bq\label{3e16a}
\left\|u\right\|_{W^{m,\infty}}
\leq
c2^m \left\|v\right\|_{W^{m,\infty}}.
\eq
\end{lem}

\begin{proof}
It is known (see, e.g., \cite{brunner2004collocation}) that there exists
a smooth {\em resolvent} kernel $R(t, s)$ depending only on $K(t,s)$
such that the solution of (\ref{3e16a}) satisfies
\[
 u(t)= v(t) + \int ^t_0 R(t,s) v(s)ds.
\]
Using the transformation $\tau=s/t$ leads to
\[
u(t)= v(t) + \int ^1_0 t R(t, t\tau) v(t\tau) d\tau.
\]
Differentiating the above equation with respect to $t$ $m$-times gives
\[
u^{(m)}(t) = v^{(m)}(t) + \sum ^{m}_{j=0} C^j_m \int^1_{0} \partial ^{(m-j)}_t
(t R(t, t\tau)) v^{(j)}(t\tau) \tau^j d\tau,
\]
which leads to
\[
\Vert u^{(m)} \Vert \le \Vert v^{(m)}\Vert + c\sum ^{m}_{j=0} C^j_m
\int^1_0  \tau^jd\tau
\Vert v^{(m)}\Vert \le c 2^m \Vert v^{(m)}\Vert ,
\]
where the constsnt $c$ depends on the regularity of $R$, or quivalently
the regularity of $K$. This completes the proof of the lemma.
\end{proof}

We also need the following discrete Gronwall lemma (see, e.g.,
\cite{brunner2004collocation}).

\begin{lem}\label{lem-gron}
Assume that  $\phi_n, p_n$ and $k_n$ are three sequences, $k_n$ is non-negative,
and they satisfy
\bex
\left\{
\begin{array}{l}
  \phi_0  \leq   g_0, \\
  \phi_n  \leq  g_0+ \sum\limits_{s=0}^{n-1} p_s +\sum\limits_{s=0}^{n-1} k_s \phi_s, \ \ n\geq 1.
\end{array}
\right.
\eex
Then it holds
\bex
\left\{
\begin{array}{l}
  \phi_1 \leq g_0(1+k_0) + p_0\\
  \phi_n \leq g_0 \prod\limits_{s=0}^{n-1} (1+k_s) + \sum\limits_{s=0}^{n-2}p_s \prod\limits_{\tau=s+1}^{n-1}(1+k_r) +p_{n-1}, \ \ n\geq 2.
\end{array}
\right.
\eex
Moreover, if $g_0\geq 0$  and $p_n \geq 0$ for all $n\geq 0$, then it holds
\bex
\phi_n \leq \left(g_0+ \sum_{s=0}^{n-1} p_s\right) \exp \left(\sum_{s=0}^{n-1}k_s\right), \ \ n\geq 1.
\eex
\end{lem}

\section{Stability and Convergence Analysis}
\setcounter{equation}{0}

This section is devoted to the stability and convergence analysis
of the proposed parallel in time scheme.

\subsection{Stability}

\begin{lem}\label{lem1-1}
For  $1\leq n \leq N$, let $\F_n$ be the fine approximation operator
defined by the spectral collocation
scheme \eqref{1e5}, $\G_n$ defined by \eqref{1e7},
$g_{M}=\I_n^{M}g$,
and  $\{\psi_i\}_{i=1}^{n-1}$ be a polynomial sequence.
Then for sufficiently large $M$ and $\tilde M$ we have
\br
&&\left\|\F_{n} (t;\psi_{1}, \cdots, \psi_{n-1})
\right\|_\infty \le c \Big( \|g_{M}\|_\infty+ \Delta t
\sum ^{n-1} _{j=1} \|\psi_j\|_\infty \Big),
\label{sta-eq20} \\
&&\left\|\G_{n} (t;\psi_{1}, \cdots, \psi_{n-1})
\right\|_\infty \le c \Big(
\|g_{\tilde{M}}\|_\infty+ \Delta t
\sum ^{n-1} _{j=1} \|\psi_j\|_\infty \Big), \;\;\; n\ge1,
\label{sta-eq21}
\er
where it is understood that $\F_1(t;\psi_0):=\F_1(t)$ and $\G_1(t;\psi_0):=\G_1(t)$.
\end{lem}

\begin{proof}
Since the only difference between $\F_{n}$ and $\G_{n}$ is
the polynomial degree,
we only need to prove  \eqref{sta-eq20}.
Let $\psi_n:=\F_{n} (t;\psi_{1}, \cdots, \psi_{n-1}) $.
It follows from the
spectral collocation scheme \eqref{1e5} that
\bq  \label{3eq1a}
\psi_n(\xi_n^i)=
g(\xi_n^i)
+ \left(\bar{K}(\xi_n^i,s_n^i), \psi_n(s_n^i)\right)_{M}
+
\frac{\Delta t}{2}\sum^{n-1}_{j=0}
\left(K(\xi_n^i, s_{j}), \psi_{j}(s_{j})\right)_{M},\ 0\le i\le M,
\eq
which can be further re-organized as
\bq \label{3eq1}
\psi_n(\xi_n^i)=
g(\xi_n^i)+\left(\bar{K}(\xi_n^i,s_n^i),
\psi_n(s_n^i)\right) +A^i_n +
\frac{\Delta t}{2}\sum^{n-1}_{j=1} \left(K(\xi_n^i, s_{j}),
\psi_{j}(s_{j})\right) +\sum^{n-1}_{j=1}  A^i_{j},
\eq
where
\ben
&& A_n^i=
-\left(\bar{K}(\xi_n^i,s_n^i), \psi_n(s_n^i)\right)
+\left(\bar{K}(\xi_n^i,s_n^i), \psi_n(s_n^i)\right)_{ M},\nn
\\
&& A_j^i=
-\frac{\Delta t}{2}\left(K(\xi_n^i,s_{j}), \psi_j(s_{j})\right)
+\frac{\Delta t}{2}\left(K(\xi_n^i,s_{j}), \psi_j(s_{j})\right)_{ M},\quad  1\leq j\leq n-1.\nn
\een
Multiplying both sides of \eqref{3eq1} by the Lagrange basis
function $h_n^i$ and summing up from $i=0$ to $i=M$, we obtain
\br
\psi_n(t)
&=&
\I_n^Mg+
\I_n^M \int_{t_{n-1}}^t K(t,s) \psi_n(s) \d s
+\sum_{i=0}^{M} A_n^i h_n^i(t)
\label{sta-eq1} \\
&&\;\;\; + \sum^{n-1}_{j=1} \Big[ \frac{\Delta t}{2}
\I_n^M \left(K(t, s_{j}), \psi_{j}(s_{j})\right)
+ \sum_{i=0}^{M} A_{j}^i h_{j}^i(t)\Big]. \nn
\er
Consequently, we have
\bq \label{eqa1}
\psi_n(t)
=
\int_{t_{n-1}}^{t}  K(t, s) \psi_n(s) \d s +\I_n^Mg+\sum_{j=1}^{n} J_j+\sum_{j=1}^{n} I_j
+\sum_{j=1}^{n-1} R_j,
\eq
where
\ben
&& J_j =
\sum\limits_{i=0}^{M}A_j^ih_j^i, \quad 1\leq j \leq n,
\\
&& I_n =
\I_n^M\left(\bar{K}(t,s_n^i), \psi_n(s_n^i)\right)
-
\left(\bar{K}(t,s_n^i), \psi_n(s_n^i)\right), \\
&& I_j=\frac{\Delta t}{2}\I_n^M\left(K(t,s_{j}), \psi_j(s_{j})\right)
-\frac{\Delta t}{2}\left(K(t,s_{j}),\psi_j(s_{j})\right),
\quad 1\leq j \leq n-1,
\\
&& R_j =
\frac{\Delta t}{2}\left(K(t, s_{j}), \psi_j(s_j)\right), \quad 1\leq j \leq n-1.
\een
It follows from the standard Gronwall inequality for (\ref{eqa1})
that
\bq \label{sta-eq8}
\|\psi_n\|_\infty
\leq  c \left(\|g_M\|_\infty +
\sum_{j=1}^{n} \|J_j\|_\infty+ \sum_{j=1}^{n}\|I_j\|_\infty+\sum_{j=1}^{n-1}\|R_j\|_\infty\right).
\eq
We now estimate the right hand-side of (\ref{sta-eq8})
term by term. First, we estimate the terms $J_j$.
Using the error estimate \eqref{3e4} for the LG-quadrature, we have
 \br
\max_{0\leq i\leq M}|A_n^i|
&=&
\max_{0\leq i\leq M}\left|\left(\bar{K}(\xi_n^i,s_n^i), \psi_n(s_n^i)\right)_{ M}
-\int_{t_{n-1}}^{\xi_n^i}  K(\xi_n^i, s) \psi_n(s) \d s\right|
\label{3e12''} \\
&\leq &
 c\Delta tM^{-1}\max_{0\leq i\leq M}
|K(\xi_n^i,s_n^i)|_{H^{1;M}}
\max_{0\leq i\leq M}\|\psi_n(s_n^i)\|\nn \\
&\leq&
c_1 \Delta tM^{-1}
\|\psi_n\|_{\infty}, \nn
\er
where $c_1$ depends on $|K(\xi_n^i,s_n^i(\cdot))|_{H^{1;M}}$.
Similarly,
 \bq
\max_{0\leq i\leq M}|A_j^i|
\leq
 c_1\Delta tM^{-1}
\|\psi_j\|_{\infty}, \quad  1\leq j\leq n-1.\label{sta-eq13}
\eq
Hence, combining  \eqref{3e12''} and \eqref{sta-eq13} with Lemma \ref{lem3e4} gives
\br
\|J_n\|_{\infty}
=
\left\|\sum_{i=0}^{M} A_n^ih_n^i\right\|_{\infty}
\leq
\max_{0\leq i\leq M}|A_n^i|\max_{x\in I_n}\sum_{i=0}^{M}|h_n^i|
\leq
c_1\Delta t M^{-\frac{1}{2}}
\|\psi_n\|_{\infty}
\leq \frac{1}{3c} \|\psi_n\|_\infty
\nn 
\er
for sufficiently small $\Delta t$ or large $M$.
Following the same lines leads to
\bq
\|J_j\|_{\infty}
\leq c_1\Delta t M^{-\frac{1}{2}} \|\psi_j\|_\infty, \quad  1\leq j \leq n-1.\label{sta-eq14}
\eq
We now estimate the second term $I_j$.
For $j=n$, applying  \eqref{3e5} leads to
\br
\label{3e22}
\left\|I_n\right\|_\infty
&=&
\left\|(\I_n^M -\I)\int_{t_{n-1}}^t K(t,s) \psi_n(s)\d  s\right\|_{\infty}
\leq
c_1 M^{-\frac{1}{4}} \left|\int_{t_{n-1}}^t K(t,s)
\psi_n(s)\d  s \right|_{H^{1;M}}  \\
&\leq&  c_1 M^{-\frac{1}{4}}  \|\psi_n\|_\infty \leq
\frac{1}{3c}\|\psi_n\|_\infty. \nn
\er
For $ 1\leq j \leq n-1$, in virtue of \eqref{3e5}, we derive
\br
\label{sta-eq15}
\left\|I_j\right\|_\infty
&=&
\left\|(\I_n^M -\I)\int_{t_{j-1}}^{t_j} K(t,s) \psi_j(s)\d s\right\|_{\infty}
\leq c_1 M^{\frac{3}{4}-1} \left|\int_{t_{j-1}}^{t_j} K(t,s)
\psi_j(s) \d s\right|_{H^{1;M}} \nn \\
&\leq&
 c_1 \Delta t M^{-\frac{1}{4}} \left\| \psi_j \right\|_\infty.
\nn \er
It remains to estimate $R_j$. By a direct calculation, we have
\br\label{sta-eq16}
\|R_j\|_\infty=\frac{\Delta t}{2}\left\|\left(K(t, s_{j}), \psi_j(s_j)\right)\right\|_\infty
\leq
c\Delta t\|\psi_{j}\|_{\infty}, \quad 1\leq j \leq n-1.
\er
Finally, combining \eqref{sta-eq8}-\eqref{sta-eq16} gives
\br
\|\psi_n\|_\infty
&\leq &
c \|g_M\|_\infty+\frac{1}{3} \|\psi_n\|_\infty+
c\Delta t M^{-\frac{1}{2}} \sum^{n-1}_{j=1}\|\psi_j\|_\infty \\
&&+
\frac{1}{3}\|\psi_n\|_\infty+
 c\Delta t M^{-\frac{1}{4}} \sum^{n-1}_{j=1}\left\| \psi_j \right\|_\infty+
c\Delta t\sum^{n-1}_{j=1}\|\psi_{j}\|_{\infty}.\nn
\er
Now a simple rearrangement leads to \eqref{sta-eq20}.
\end{proof}

We are now ready to state and prove one of the main results of this paper.

\begin{thm}\label{thm1-1}(Stability)
The iteration scheme \eqref{n2e5}-\eqref{n2e5a}
is stable in the sense that the solution $U_n^k$ satisfies
\bq
\label{stab}
\|U_n^k\|_\infty \leq c\|g_{\tilde M}\|_\infty, \;\;\;  \forall k\geq 0, \;\;
1\leq  n \leq N.
\eq
\end{thm}

\begin{proof}
For $k=0$, it follows from the initial value
$U^0_n = \G_n(t; U^0_1, \cdots, U^0_{n-1})$ and Lemma \ref{lem1-1}
that
\bq
\|U_n^0\|_\infty
=\left\|\G_{n} \left(t,U_{1}^0, \cdots,
U_{n-1}^0\right)\right\|_\infty
\leq c_0\Big(\|g_{\tilde M}\|_\infty+ \Delta t \sum^{n-1}_{j=1}
\|U_{j}^0\|_\infty \Big).
\eq
By applying the discrete Gronwall Lemma \ref{lem-gron}, we get
\bq
\|U_n^0\|_\infty
\leq   c_0e^{c_0(n-1)\Delta t}\|g_{\tilde M}\|_\infty
\leq
c_0e^{c_0T}\|g_{\tilde M}\|_\infty.
\label{sta-eq10}
\eq
For $k\geq 1$, according to the iteration scheme
\eqref{n2e5}-(\ref{n2e5a}), we have
\br
\|U_n^k \|_\infty
&\leq&
\left\|
 \G_{n} \left(t;U_{1}^k, U_{2}^k, \cdots, U_{n-1}^k\right)
\right\|_\infty \nn\\
&&+
\left\|\F_{n} \left(t;U_{1}^{k-1}, U_{2}^{k-1}, \cdots,
U_{n-1}^{k-1}\right)\right\|_\infty+
\left\|\G_{n} \left(t;U_{1}^{k-1}, U_{2}^{k-1}, \cdots,
U_{n-1}^{k-1}\right)\right\|_\infty.
\label{con-eq1}
\er
Applying  Lemma \ref{lem1-1}  to the right hand side of \eqref{con-eq1} yields
\br
\|U_n^k\|_\infty
&\leq &
c \Big(\|g_{\tilde M}\|_\infty+\Delta t \sum^{n-1}_{j=1}
\|U_{j}^k\|_\infty \Big)
+ c \Big(\|g_{\tilde{M}}\|_\infty+\|g_{M}\|_\infty+
\Delta t \sum^{n-1}_{j=1} \|U_{j}^{k-1}\|_\infty \Big)
\nn \\
&\leq &
c_0 \Delta t \sum^{n-1}_{j=1}
\|U_{j}^k\|_\infty
+ c \Big(\|g_{\tilde M}\|_\infty+
\Delta t \sum^{n-1}_{j=1} \|U_{j}^{k-1}\|_\infty \Big).
\label{sta-eq11}
\er
In the following, we will derive the following inequality
\br
\|U_n^k\|_\infty
\leq
c_0
\sum_{l=0}^k
\frac{e^{c_0(l+1)(n-1)\Delta t}}{l!}
\|g_{\tilde M}\|_\infty, \ \  \forall k\geq 0, 1\leq  n \leq N,
\label{stabx}
\er
where $c_0$ is a constant independent of $k$.
We do this by induction.
First, it follows from \eqref{sta-eq10} that (\ref{stabx}) is true for $k=0$.
We now show  that if (\ref{stabx}) holds for a given $k$
then it also holds for $k+1$.
It follows from \eqref{sta-eq11} (replace $k$ by $k+1$)
and the discrete Gronwall Lemma \ref{lem-gron} that
\bq
\|U_n^{k+1}\|_\infty \le
c_0 e^{c_0 (n-1)\Delta t}  \left(\|g_{\tilde M}\|_\infty+
\Delta t \sum^{n-1}_{j=1} \|U_{j}^{k}\|_\infty \right).
\eq
Then using the induction assmption  \eqref{stabx} gives
\br
\|U_n^{k+1}\|_\infty &\leq&
c_0 \|g_{\tilde M}\|_\infty e^{c_0 (n-1)\Delta t}
\left(1+c_0\Delta t
\sum_{l=0}^k
\frac{e^{c_0(l+1)(n-2)\Delta t}}{l!}
+
\cdots+ c_0\Delta t\sum_{l=0}^k
\frac{1}{l!}\right)\nn\\
&=&
c_0\|g_{\tilde M}\|_\infty e^{c_0 (n-1)\Delta t}
\left(1+c_0\Delta t
\sum_{l=0}^k
\frac{1}{l!}\frac{e^{c_0(l+1)(n-1)\Delta t}-1}{e^{c_0 (l+1)\Delta t}-1}
\right)\nn\\
&\leq&
c_0 \|g_{\tilde M}\|_\infty e^{c_0 (n-1)\Delta t}
\left(1+
\sum_{l=0}^k
\frac{1}{(l+1)!}(e^{c_0(l+1)(n-1)\Delta t}-1)
\right)\nn\\
&\leq&
c_0 \|g_{\tilde M}\|_\infty e^{c_0 (n-1)\Delta t}
\left(1+
\sum_{l=0}^k
\frac{1}{(l+1)!}e^{c_0(l+1)(n-1)\Delta t}
\right)\nn\\
&\leq&
c_0 \|g_{\tilde M}\|_\infty
\sum_{l=0}^{k+1}
\frac{e^{c_0(l+1)(n-1)\Delta t}}{l!}.
\nn  \er
This implies that \eqref{stabx} also holds for $k$ being replaced by $k+1$.
Finally, by noticing the fact that
\bq
\label{4e16a}
\sum_{l=0}^{\infty}
\frac{e^{c_0(l+1)(n-1)\Delta t}}{l!}
\leq
e^{c_0(n-1)\Delta t} e^{e^{c_0(n-1)\Delta t}}
\leq
e^{c_0T} e^{e^{c_0T}}, \ \  1\leq  n \leq N,
\eq
the desired result (\ref{stab}) follows from (\ref{stabx}).
\end{proof}

\vskip .25cm
{\bf Remark.} It is seen from (\ref{stabx}) and (\ref{4e16a})
that the generic constant $c$ in (\ref{stab}) has a rapid,
double-exponential growth in $T$. Theoretically, this may
become very large. On the other hand, our numerical experiemnts
(see Section 5) can take $T=100$ without any apparent problem.
It may be possible to obtain a better generic constant
in (\ref{stab}), even by assuming something more on
the kernel function $K(t, s)$.  
We believe that other proof techniques are needed to improve this
double-exponential growth constant, and this is certainly an
interesting theoretical challenge.

\subsection{Convergence}

We now provide an estimate for the fine approximation
operator $\F_n$.

\begin{lem}\label{lem00}
Let $u_n$  be the solution of \eqref{n2e1} and $U_n$  be
the solution of \eqref{n2e4}.
 If  $u \in H^m(I), m\geq 1$, then we have
\bq
\|u_{n}-U_n\|_\infty
\leq
c\left( M^{\frac{3}{4}-m}|u|_{H^{m;M}}+
M^{\frac{1}{2}-m} \|u\|_{\infty}\right), \quad 1\le n\le N,
\eq
provided the number of collocation points $M$ is sufficiently large.
\end{lem}

\begin{proof}
Replace $t$ in \eqref{n2e1} by $\zeta ^i_n$.
Subtracting the resulting equation from (\ref{1e5}) gives
\br
&& u_n(\xi_n^i)-U_n(\xi_n^i) =
\left(\bar{K}(\xi_n^i,s_n^i), u_n(s_n^i)\right) \\
&& \;\;\; -\left(\bar{K}(\xi_n^i,s_n^i), U_n(s_n^i)\right)_{ M}
+\frac{\Delta t}{2} \sum^{n-1}_{j=1}
\Big[ \left(K(\xi_n^i, s_{j}), u_{j}(s_{j})\right)-
\left(K(\xi_n^i, s_{j}), U_{j}(s_{j})\right)_{M}\Big].
\nn \er
It can be further re-organized as
\[
 u_n(\xi_n^i)-U_n(\xi_n^i) =
\left(\bar{K}(\xi_n^i,s_n^i), e_n(s_n^i)\right) + A_n^i
 +\frac{\Delta t}{2}\sum^{n-1}_{j=1} \Big[
\left(K(\xi_n^i, s_{j}),
e_{j}(s_{j})\right) +A^i_{j} \Big],
\]
where $e_j := u_j - U_j$,
\br
&& A_n^i=
\left(\bar{K}(\xi_n^i,s_n^i), U_n(s_n^i)\right)
-\left(\bar{K}(\xi_n^i,s_n^i), U_n(s_n^i)\right)_{ M},\nn
\\
&& A_j^i=
\frac{\Delta t}{2}\left(K(\xi_n^i,s_{j}), U_j(s_{j})\right)
-\frac{\Delta t}{2}\left(K(\xi_n^i,s_{j}), U_j(s_{j})\right)_{M},\quad  1\leq j\leq n-1.\nn
\er
Following the same procedure as in the proof of Lemma 4.1
gives
\begin{eqnarray*}
\I_n^M u_n-U_n
=
\I_n^M\left(\bar{K}(\xi_n^i,s_n^i), e_n(s_n^i)\right)
+\sum_{i=0}^{M}A_n^ih_n^i
+ \sum^{n-1}_{j=1} \Big[
\frac{\Delta t}{2}\I_n^M\left(K(\xi_n^i, s_{j}), e_{j}(s_{j})\right)
+\sum_{i=0}^{M}A_{j}^ih_{j}^i\Big].
\end{eqnarray*}
Consequently,
\br
u_n-U_n
=
\int_{t_{n-1}}^{\xi_n^i}K(\xi_n^i,s) e_n(s) \d s
+u_n-\I_n^M u_n+
\sum_{j=1}^{n} J_j+ \sum_{j=1}^{n}I_j+\sum_{j=1}^{n-1}R_j\nn,
\er
where
\ben
&& J_j =
\sum\limits_{i=0}^{M}A_j^ih_j^i, \quad 1\leq j \leq n;
\quad
R_j =
\frac{\Delta t}{2}\left(K(\xi_n^i, s_{j}), e_j\right), \quad 1\leq j \leq n-1,
\\
&& I_n =
\I_n^M\left(\bar{K}(\xi_n^i,s_n^i), e_n(s_n^i)\right)
-
\left(\bar{K}(\xi_n^i,s_n^i), e_n(s_n^i)\right),  \\
&& I_j=\frac{\Delta t}{2}\I_n^M\left(K(\xi_n^i,s_{j}), e_j(s_{j})\right)
- \frac{\Delta t}{2}\left(K(\xi_n^i,s_{j}), e_j(s_{j})\right),
\quad 1\leq j \leq n-1.
\een
Applying the Gronwall Lemma \ref{lem-gron} gives
\br
\|u_{n}-U_n\|_\infty
\leq
c\left(\left\|\I_n^M u_n-u_n\right\|_\infty+
\sum_{j=1}^{n} \|J_j\|_\infty+ \sum_{j=1}^{n}\|I_j\|_\infty
+ \sum_{j=1}^{n-1}\|R_j\|_\infty\right).\label{con-eq17}
\er
We now estimate the right-hand side of  \eqref{con-eq17} term by term.
First, it follows from the inequality \eqref{3e5} that
\br \label{3e18'}
\left\|\I_n^M u_n-u_n\right\|_\infty
\leq
 cM^{\frac{3}{4}-m} |u_n|_{H^{m;M}}.
\er
Then by using a similar technique as in the proof of Lemma \ref{lem1-1},
we can estimate  $\|I_j\|_\infty$ and $\|J_j\|_\infty$
as follows. For $j=n$ and sufficiently large $M$,
\br\label{con-eq18}
&& \|J_n\|_\infty =
\left\| \sum\limits_{i=0}^{M}A_n^ih_n^i\right\|_\infty
\!
\leq c_1\Delta tM^{\frac{1}{2}-m}
 \Big(\|e_n\|_{\infty}+\|u_n\|_{\infty}\Big) \leq
\frac{1}{3c}\|e_n\|_\infty+c\Delta tM^{\frac{1}{2}-m}\|u_n\|_{\infty},
\er
where, as in Lemma \ref{lem1-1}, $c_1$, and therefore $c$, depend on $|K(\xi_n^i,s_n^i(\cdot))|_{H^{m;M}}$.
Similarly, we have
\br\label{con-eq4}
\|I_n\|_\infty
=
\left\| \I_n^M\left(\bar{K}(\xi_n^i,s_n^i), e_n(s_n^i)\right)
-
\left(\bar{K}(\xi_n^i,s_n^i), e_n(s_n^i)\right)\right\|_\infty
\leq
\frac{1}{3c}\|e_n\|_\infty.
\er
For $1 \leq j \leq n-1$, we can further have
\bsub \label{con-eq6}
\begin{align}
& \|J_j\|_\infty =
\left\|\sum\limits_{i=0}^{M}A_j^ih_j^i\right\|_\infty
\leq
c\Delta tM^{\frac{1}{2}-m}
\Big(\|e_j\|_{\infty}+\|u_j\|_{\infty}\Big),
\label{con-eq6a} \\
& \|I_j\|_\infty
= \frac{\Delta t}{2}\left\|\I_n^M\left(K(\xi_n^i,s_{j}), e_j(s_{j})\right)
- \left(K(\xi_n^i,s_{j}), e_j(s_{j})\right)\right\|_\infty
\leq
 c\Delta t\|e_j\|_\infty,
\label{con-eq6b} \\
&\|R_j\|_\infty=\frac{\Delta t}{2}\left\|\left(K(t, s_{j}), e_j(s_j)\right)\right\|_\infty
\leq
c\Delta t\|e_{j}\|_{\infty}.
\label{con-eq6c}
\end{align}
\esub
Combining \eqref{con-eq17}-\eqref{con-eq6} yields
\br
\|e_n\|_\infty
&\leq& c \Big(M^{\frac{3}{4}-m} |u_n|_{H^{m;M}}+\frac{2}{3c}\|e_n\|_\infty
+\Delta tM^{\frac{1}{2}-m}  \|u_{n}\|_\infty \Big)\nn\\
&& \;\; + c\Delta tM^{\frac{1}{2}-m} \sum^{n-1}_{j=1}
\|u_{j}\|_{\infty}
+c\Delta t \sum^{n-1}_{j=1} \|e_{j}\|_\infty.
\nn \er
Using the discrete Gronwall Lemma \ref{lem-gron} gives
\br
\|e_n\|_\infty
\leq
c\left( M^{\frac{3}{4}-m}|u|_{H^{m;M}}+
n\Delta tM^{\frac{1}{2}-m} \|u\|_{\infty}\right)e^{c(n-1)\Delta t}.\label{eq-7}
\er
Finally the lemma is proved by observing
that $n\Delta t$ and $e^{c(n-1)\Delta t}$
can be bounded by a constant depending on $T$.
\end{proof}

By following the same lines as in the proof of Lemma \ref{lem1-1},
we can establish the continuity of the approximation operators
$\F_n$ and  $\G_n$.

\begin{lem}\label{lem1}
For $1\leq n\leq N$, both operators $\F_n$ and  $\G_n$ are continuous,
i.e.,  for any two polynomial
sequences $\{\psi_i\}_{i=1}^{n-1}$
and $\{\varphi_i\}_{i=1}^{n-1}$,  we have
\br
&&\|\F_n (t;\psi_{1}, \cdots, \psi_{n-1}) -
\F_n (t;\varphi_{1}, \cdots, \varphi_{n-1})\|_\infty
\leq  c\Delta t \sum ^{n-1} _{j=1} \|\psi
_{j}- \varphi_{j}\|_\infty;
 \\
&&\|\G_n (t;\psi_{1}, \cdots, \psi_{n-1}) -
\G_n (t,\varphi_{1}, \cdots, \varphi_{n-1})\|_\infty
\leq c\Delta t \sum^{n-1}_{j=1} \|\psi_{j}- \varphi_{j}\|_\infty.
\er
\end{lem}

We now define an auxiliary operator $\tilde\E_n$. For a
sequence  $\{\psi_i\}_{i=1}^{n-1}$,  we define
$\tilde\E_n(t;\psi_{1}, \psi_2, \cdots, \psi_{n-1})$
as the function $\psi_n$,
which is the solution of the following problem:
\bq\label{con-eq31}
\psi_n(t)- \ds\int_{t_{n-1}}^tK(t,s) \psi_n(s) \d s
=g(t) +\frac{\Delta t}{2} \sum^{n-1}_{j=1}
\left(K(t,s_{j}), \psi_{j}(s_{j})\right)_{M}.
\eq

\begin{lem}\label{lem2}
For $1\leq n\leq N$,  let 
$\delta\F_n=\tilde\E_n-\F_n $, $\delta\G_n=\tilde\E_n-\G_n$. Then they are continuous in the
sense that they satisfy, for any two
sequences $\{\psi_i\}_{i=1}^{n-1}$ and $\{\varphi_i\}_{i=1}^{n-1}$,
\br
&&\left|\delta \F_n (t;\psi_{1}, \cdots, \psi_{n-1})
- \delta \F_n (t;\varphi_{1}, \cdots, \varphi_{n-1})\right|
\leq c \Delta t 2^m M^{\frac{3}{4}-m} \sum^{n-1}_{j=1} \|\psi
_{j}- \varphi_{j}\|_\infty ; \\
&&|\delta \G_n (t; \psi_{1},  \cdots, \psi_{n-1})
- \delta \G_n (t; \varphi_{1}, \cdots, \varphi_{n-1})|\leq
c \Delta t 2^m \tilde{M}^{\frac{3}{4}-m} \sum^{n-1}_{j=1} \|\psi
_{j}- \varphi_{j}\|_\infty ,
\er
where $c$ depends on
$\max_{s\in I}\left\|K(\cdot, s)\right\|_{W^{m,\infty}}$,
$m\ge 1$.
\end{lem}

\begin{proof}
Similar to the proof of  Lemma \ref{lem1-1}, we only need to prove
the first inequality above.
Let $q_n(t)=\F_n (t;\psi_{1},\cdots, \psi_{n-1}),
p_n(t)= \F_n (t;\varphi_{1}, \cdots, \varphi_{n-1}),
v_n(t)=
 \tilde\E_n (t;\psi_{1},\cdots, \psi_{n-1}),
w_n(t)= \tilde\E_n (t; \varphi_{1}, \cdots, \varphi_{n-1}).$
Then
\bq
\delta \F_n (t;\psi_{1}, \cdots, \psi_{n-1})-
\delta \F_n (t;\varphi_{1},\cdots, \varphi_{n-1})
= (v_n-q_n)-(w_n-p_n).
\eq
Let $e_l=\psi_{l}-\varphi_{l}$. From  the
definitions \eqref{con-eq31} and \eqref{1e5}, we can verify that
\br
&&(v_n-w_n)(\xi_n^i)-(q_n-p_n)(\xi_n^i)\label{con-eq32} \\
&=&
\left(\bar{K}(\xi_n^i,s_n^i), (v_n-w_n)(s_n^i)\right)
-\left(\bar{K}(\xi_n^i,s_n^i), (q_n-p_n)(s_n^i)\right)_{ M}.
\nn
\er
Let $\Delta_n=(v_n-q_n)-(w_n-p_n)$. By multiplying both sides of
equation  \eqref{con-eq32} by $h_n^i$ and summing up from $i=0$ to $i=M$, we obtain
\br\label{con-eq28}
\I_n^M (v_n-w_n)-(q_n-p_n)
=
\I_n^M\left(\bar{K}(t,s_n^i), \Delta_n(s_n^i)\right)
+\sum_{i=0}^{M}A_n^ih_n^i,
\nn
\er
where
\bq
A_n^i
=
\left(\bar{K}(\xi_n^i,s_n^i), (q_n-p_n)(s_n^i)\right)
-\left(\bar{K}(\xi_n^i,s_n^i), (q_n-p_n)(s_n^i)\right)_{ M}.\nn
\eq
Consequently,
\br
\Delta_n(t)
=
\int_{t_{n-1}}^t K(t,s) \Delta_n(s)\d s+\I_n^M (v_n-w_n)-(v_n-w_n)
+J_n+I_n,\label{con-eq9}
\er
where
\bq
J_n =
\sum\limits_{i=0}^{M}A_n^ih_n^i,
\quad
I_n =
\I_n^M\left(\bar{K}(t,s_n^i), \Delta_n(s_n^i)\right) -
\left(\bar{K}(t,s_n^i), \Delta_n(s_n^i)\right).
\nn \eq
Applying the standard Gronwall inequality to \eqref{con-eq9} gives
\bq
\|\Delta_n\|_\infty
\leq
c \Big( \left\|\I_n^M (v_n-w_n)-(v_n-w_n)\right\|_\infty
+\|J_n\|_\infty+\|I_n\|_\infty\Big).
\label{sta-eq23}
\eq
It follows from the inequality \eqref{3e5} that
\br
\left\|\I_n^M (v_n-w_n)-(v_n-w_n)\right\|_\infty
\leq
 c M^{\frac{3}{4}-m}
|v_n-w_n|_{H^{m;M}}
\leq
c M^{\frac{3}{4}-m}
\|v_n-w_n\|_{W^{m,\infty}}.
\nn 
\er
Note that
\bq
v_n-w_n=
\left(\bar{K}(t, s_{n}), (v_n-w_n)(s_{n})\right) +
\frac{\Delta t}{2} \sum_{j=1}^{n-1}
\left(K(t, s_{j}), e_{j}(s_{j})\right)_{M}.
\eq
Consequently, it follows from Lemma \ref{lem3} that
\br
\label{sta-eq22}
&&\|v_n-w_n\|_{W^{m,\infty}}\leq
c 2^m \sum ^{n-1}_{j=1} \|\left(K(t, s_{j}),
e_{j}(s_{j})\right)_M\|_{W^{m,\infty}}  \\
&\leq& c 2^m \Delta t \max_{s\in I}\left\|K(\cdot, s)
\right\|_{W^{m,\infty}}
\sum^{n-1}_{j=1}\|e_{j}\|_\infty
\leq
c 2^m \Delta t \sum^{n-1}_{j=1}\|e_{j}\|_\infty.
\nn \er
Using Lemma \ref{lem1} gives
\bq
\|J_n\|_\infty =
\left\| \sum\limits_{i=0}^{M}A_n^ih_n^i\right\|_\infty
\leq
 c\Delta tM^{\frac{1}{2}-m}
 \|q_n-p_n\|_{\infty} \leq  c
 \Delta t^{2}M^{\frac{1}{2}-m} \sum^{n-1}_{j=1}\|e_{j}\|_\infty.
\eq
Moreover, using the same technique as used in the proof og
Lemma \ref{lem1-1} yields
\br\label{con-eq29}
\|I_n\|_\infty
=
\left\| \I_n^M\left(\bar{K}(t,s_n^i), \Delta_n(s_n^i)\right)
-
\left(\bar{K}(t,s_n^i), \Delta_n(s_n^i)\right)\right\|_\infty
\leq
\frac{1}{3c} \|\Delta_n\|_\infty.
\er
Combining \eqref{sta-eq23}-\eqref{con-eq29}, we conclude that
\[
\|\Delta_n\|_\infty
\leq
c\Delta t 2^m M^{\frac{3}{4}-m} \sum^{n-1}_{j=1}\|e_{j}\|_\infty.
\]
The proof is then complete.
\end{proof}

\begin{thm}\label{th1} (Convergence)
For $1\leq n\leq N$,  let $u_n$  be the solution of  \eqref{n2e1}
and $U_n^k$  be the solution of iteration scheme
\eqref{n2e5}-(\ref{n2e5a}).
 If
 $\max_{s\in I}\left\|K(\cdot, s)\right\|_{W^{m,\infty}}\le c$,
 $u \in H^m(I), m\geq 1$, then
 \bq \label{conv}
\|u_n-U_n^k\|_\infty \leq c \Big( M^{\frac{3}{4}-m} + (\tilde
M/2)^{(\frac{3}{4}-m)(k+1)}\Big)\Big(|u|_{H^{m;M}}+\|u\|_\infty\Big),
\;\;\; k=0,1,\cdots ,
\eq
where
$c$ depends on $T$ and
$\max_{s\in I}\left\|K(\cdot, s)\right\|_{W^{m,\infty}}$.
\end{thm}

\begin{proof}
We first prove that (\ref{conv}) is true for $k=0$.
By construction, we have 
$U^0_n = \G_n(t; U^0_1, \cdots, U^0_{n-1})$.  
Applying Lemma \ref{lem00} to the above coarse solution gives
\br
\|u_{n}-U^0_n\|_\infty
\leq
c_1\left( {\tilde M}^{\frac{3}{4}-m}|u|_{H^{m;M}}+
{\tilde M}^{\frac{1}{2}-m} \|u\|_{\infty}\right)
&\leq&
c_2  {\tilde M}^{\frac{3}{4}-m} \Big(|u|_{H^{m;M}}+ \|u\|_{\infty}\Big)
\nn\\
&\leq&
c  ({\tilde M/2})^{\frac{3}{4}-m} \Big(|u|_{H^{m;M}}+ \|u\|_{\infty}\Big).
\label{te1}
\er
This proves  (\ref{conv}) for $k=0$.
We now prove (\ref{conv}) for $k\ge 1$.
It follows from \eqref{n2e4} and \eqref{n2e5}
that
\br
\left|U_n^k- U_n\right|
&\leq&
\left|\G_n \left(t;U_{1}^k, \cdots, U_{n-1}^k\right) -
\G_n \left(t,U_{1},  \cdots, U_{n-1}\right)\right|\nn\\
&&\;\; + \left|\delta \G_n \left(t; U_{1}^{k-1}, \cdots, U
_{n-1}^{k-1}\right)
-\delta \G_n (t,U_{1}, \cdots, U_{n-1})\right|
\nn \\
&& \;\; + \left|\delta \F_n \left(t,U_{1}^{k-1},\cdots, U_{n-1}^{k-1}\right)
- \delta \F_n (t,U_{1}, \cdots, U_{n-1})\right|,\nn
\er
where $\delta\F_n=\tilde\E_n-\F_n$ and $\delta\G_n=\tilde\E_n-\G_n$
with $\tilde\E_n$ defined by \eqref{con-eq31}.
Using Lemmas \ref{lem1} and \ref{lem2} gives
\bq \label{xx1}
\|U_n^k- U_{n}\|_\infty \leq
c \Delta t \sum^{n-1}_{j=1} \|U_{j}^k- U_{j}\|_\infty
+c\varepsilon \Delta t \sum^{n-1}_{j=1}\|U_{j}^{k-1}- U_{j}\|_\infty,
\eq
where $\varepsilon=2^m \tilde M^{\frac{3}{4}-m}$,
and we have used the fact that
$M^{\frac{3}{4}-m}\le \tilde M^{\frac{3}{4}-m}$.
Next we will derive the following inequality
\bq
\|U_n^k-
U_n\|_\infty \leq \frac{c}{k!} e^{c(k+1)(n-1)\Delta t}
\varepsilon^{k+1}(|u|_{H^{m;M}}+\|u\|_{\infty}), \ \forall k\ge 1,
\label{eq-8}
\eq
where $c$ is a constant independent of $k$. We do this by induction.
For $k=1$, using Gronwall inequality to \eqref{xx1} gives
\br
\|U_n^1- U_n\|_\infty 
&\leq& 
c\Delta t \varepsilon e^{c(n-1)\Delta t}\sum_{j=1}^{n-1} \|U_{j}^{0} - U_{j}\|_\infty \nn \\
&\leq &
c\Delta t \varepsilon e^{c(n-1)\Delta t}\sum_{j=1}^{n-1} 
(\|U_{j}^{0} - u_{j}\|_\infty + \|u_{j} - U_{j}\|_\infty). \nn 
\er
Using \eqref{te1} and \eqref{eq-7} yields
\br
\|U_n^1- U_n\|_\infty 
&\leq&
c^2\Delta t\varepsilon^2 e^{c(n-1)\Delta t}\sum_{j=1}^{n-1}e^{c(j-1)\Delta t} (|u|_{H^{m;M}}+\|u\|_{\infty})\nn\\
&\leq&
c^2\Delta t \varepsilon^2 e^{c(n-1)\Delta t} \frac{e^{c(n-1)\Delta t}-1}{e^{c \Delta t}-1}(|u|_{H^{m;M}}+\|u\|_{\infty}) \nn\\
&\leq& c e^{2c(n-1)\Delta t}\varepsilon^2(|u|_{H^{m;M}}+\|u\|_{\infty}).\nn \er
This means \eqref{eq-8} holds for $k=1$.
Now assuming \eqref{eq-8} holds for a given $k$,
we want to prove that it also holds for $k+1$.
Again, replacing $k$ in (\ref{xx1}) by $k+1$ and using Gronwall
inequality for the resulting inequality yield
\bq
\|U_n^{k+1}- U_n\|_\infty \leq
c\Delta t \varepsilon e^{c(n-1)\Delta t}
\sum^{n-1}_{j=1}  \|U_{j}^{k} - U_{j}\|_\infty.
\eq
Using the induction assumption for index $k$, i.e., (\ref{eq-8}), gives
\br
\|U_n^{k+1}- U_n\|_\infty
&\leq&
\frac{c^2}{k!} \Delta t\varepsilon e^{c(n-1)\Delta t}\sum_{i=1}^{n-1}
\varepsilon^{k+1} e^{c(k+1)(i-1)\Delta t}(|u|_{H^{m;M}}+\|u\|_{\infty}) \nn\\
&\leq&
\frac{c^2}{k!}\Delta t e^{c(n-1)\Delta t}
\frac{e^{c(k+1)(n-1)\Delta t}-1}{e^{c(k+1)\Delta t}-1} \varepsilon^{k+2}(|u|_{H^{m;M}}+\|u\|_{\infty})\nn\\
&\leq& \frac{c}{(k+1)!} e^{c(k+2)(n-1)\Delta t}
\varepsilon^{k+2}(|u|_{H^{m;M}}+\|u\|_{\infty}).\nn
\er
Thus we have proved \eqref{eq-8} for all $k\geq 1$.
Then as in the proof of Theorem \ref{thm1-1}, the boundedness of
$e^{c(k+1)(n-1)\Delta t}/k!$ implies that
\bq
\|U_n^{k}- U_n\|_\infty \le
c\varepsilon^{k+1(}|u|_{H^{m;M}}+\|u\|_{\infty}) 
\le c(\tilde M/2)^{(\frac{3}{4}-m)(k+1)}(|u|_{H^{m;M}}+\|u\|_{\infty}), \ \forall k\ge 1.
\eq
Finally, by combining the above estimate with Lemma \ref{lem00}
leads to the desired result \eqref{conv}.
\end{proof}

\section{Numerical Results}
\setcounter{equation}{0}

 We first give some implementation details of the
parallel in time scheme \eqref{n2e5}-\eqref{n2e5a}.
The overall algorithm is described below:

\medskip
{\bf Algorithm} ({\bf A1})

{\sf
\begin{itemize}
\item
 Initialization ($k=0$): $\{U_{1}^{0}(t),
\cdots, U_{N}^{0}(t)\}$, given by solving
$\G_{n} (t,U_{1}^0, \cdots, U_{n-1}^0)$
{\it successively} for $n=1,\cdots,N$.
\item
 At step $k$: Suppose known $\{U_{1}^{k-1}(t),
\cdots, U_{N}^{k-1}(t)\}$. For $n=1\cdots, n$
\begin{enumerate}
\item
  solve the fine problem $\F_{n} (t;U_{1}^{k-1}, \cdots,
U_{n-1}^{k-1})$
{\it simultaneously};
\item
  solve the coarse problem $\G_{n} (t,U_{1}^{k}, \cdots, U_{n-1}^{k})$ 
  {\it successively};
 \item
  $\G_{n} (t,U_{1}^{k-1}, \cdots, U_{n-1}^{k-1})$,
are available from the previous step.
\end{enumerate}
\item
 Update: Using \eqref{n2e5}-\eqref{n2e5a}
to update $\{U_{1}^{k}(t), \cdots, U_{N}^{k}(t)\}$.
\end{itemize}
}

\medskip
Obviously, the most expensive part of this algorithm is to solve
the fine problem
$\F_{n} (t;U_{1}^{k-1}, \cdots,$ $U_{n-1}^{k-1})$.
However, this can be handled by parallel realization.
Next we describe
the matrix form of the linear system associated with
$\F_{n} (t; U_{1}^{k-1}, \cdots, U_{n-1}^{k-1})$,
defined over the subinterval $I_n$.

Let $p^k_n=\F_{n} (t,U_{1}^{k-1},
\cdots, U_{n-1}^{k-1})$.
Then by definition \eqref{1e5}, $p_n^k\in \P_M(I_n)$ satisfies, for all $0\le i \le M$,
\begin{equation}
\label{eq-1}
 p^k_n(\xi_n^i)-
\left(\bar{K}(\xi_n^i,s_n^i), p_n^k(s_n^i)\right)_{ M}
= g(\xi_n^i) +\frac{\Delta t}{2} \sum ^{n-1}_{j=1}
\left(K(\xi_n^i,s_{j}), U^{k-1}_{j}(s_{j})\right)_{M}.
\end{equation}
Using the notations developed in the previous sections, we obtain,
for all $0\le i \le M$:
\begin{equation}
 \label{t2e2}
 p^k_n(\xi_n^i)- \sum_{j=0}^M p_n^k(\xi_n^j)
\left( \bar{K}(\xi_n^i,s_n^i), h_n^j(s_n^i)\right)_{M}
=g(\xi_n^i) +\frac{\Delta t}{2} \sum^{n-1}_{j=1}
\left(K(\xi_n^i,s_{j}), U^{k-1}_{j}(s_{j})\right)_{M}.
\end{equation}
which can be rewritten under the matrix form:
\bq\label{linsys}
({\bi I}-{\bi A}) {\bf p}_n^k={\bi f}_n^k,\quad n=1,\cdots,N.
\eq
In our calculations, the linear system \eqref{linsys} is solved by
the Gauss-Seidel iterative method with the initial guess
$U_n^{k-1}$, obtained at the previous step in Algorithm ({\bf A1}).

The coarse problems $
\G_{n} (t; U_{1}^k, \cdots, U_{n-1}^k)$
can be solved in a similar way.
Note that solving the coarse problems is
strictly sequential with respect to $n$, but
as $\tilde M$ is much smaller than $M$, the cost of this part
is relatively inexpensive.

Below we will present some numerical results obtained by the
proposed parallel in time scheme.

\begin{example} \label{ex51}
Linear Volterra integral equation with a regular kernel:
\bq
 u(t)  + \int_{0}^{t}
\sin(\pi (t-s)) u(s)\d s = g(t), \qquad 0\leq t\leq 100.
\eq
We take
\bex
g(t) = (1-1/2\pi)\sin(\pi t)-\cos(\pi t)/2\pi
\eex
such that the exact solution
  $ u(t) = \sin (\pi t)$.
\end{example}

The main purpose is to investigate the convergence behavior of
numerical solutions with respect to the polynomial degrees $M$ and iteration number $k$.
In  Figure  \ref{fig1},
we plot the $L^\infty$-errors in semi-log scale as a function of $M$,
with $\tilde M$ fixed to 13 and $N$ fixed to 20.
That is,  the domain is partitioned into 20 subintervals and the
coarse algorithm is solved with
13 collocation points in each sub-interval.
To separate different error sources, the solution is iterated
with sufficiently  large number of $k$ so that the error
$\tilde M^{({3\over 4}-m)(k+1)}$ (see \eqref{conv})  is negligible as compared
with the error of the fine resolution.
As expected, the errors show an exponential decay, since in this semi-log
representation one observes that the error variations are
linear versus the degrees of polynomial $M$.

Next we investigate the convergence behavior with respect to the iteration number $k$,
which is more interesting to us.
For a similar reason mentioned above, we now fix a large enough $M=25$,
and let $k$ vary for different values of $\tilde{M}$.
In Figure \ref{fig2}, we plot the error decay with increasing iteration number $k$
for several values of $\tilde{M}$.
It is observed that the errors decay quickly with increasing $k$,
and for ${\tilde M}= 13$ only 6 iterations are required
to reach the machine accuracy.

In the stability analysis, see Lemma \ref{lem1-1}, 
we have assumed sufficiently large $M$ and $\tilde M$. In practical calculations 
it is interesting to see 
how large these polynomial degrees are really needed to guarantee the stability of
the scheme.
In Figure \ref{fig2b}, we present the error behavior as a function of the iteration number $k$ for
$M=13, N=20, \tilde M=4$, and $T=20$. It is observed that the error decays to the machine accuracy
within about 10 iterations.
This result shows that even with moderate $M$ and small 
$\tilde M$, the proposed parallel in time scheme remains convergent.

Finally to test the sharpness of the estimate given in \eqref{conv}, we plot
the errors with respect to the degree of freedom $\tilde{M}$ for three values $k = 2, 3$, and 4 
with $M = 25$ in Figure \ref{fig3}.  
As for an usual spectral method for analytical solutions, see, e.g. \cite{canuto2006spectral}, 
we expect from \eqref{conv} an second error term behaves like $e^{-c\tilde M(k+1)}$ with a fixed 
constant $c$.
The result presented in Figure \ref{fig3} seems to be in a good agreement with the theoretical
prediction,  since in this semi-log representation the error curves for different $k$ are all straight lines with 
slopes corresponding to the constant $c=1.4$.

\begin{figure}[!htb]
\centering
\includegraphics[width=10cm,height=6.9cm]{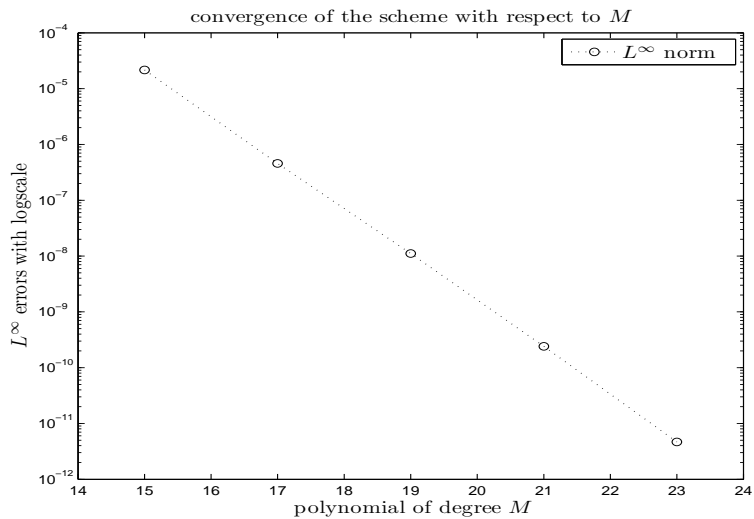}
\caption{$L^\infty $-errors versus the degree of
freedom for fine approximation $M$ with $\tilde{M}=13, N=20$.}
\label{fig1} 
\end{figure}

\begin{figure}[!htb]
\centering
\includegraphics[width=10cm,height=6.9cm]{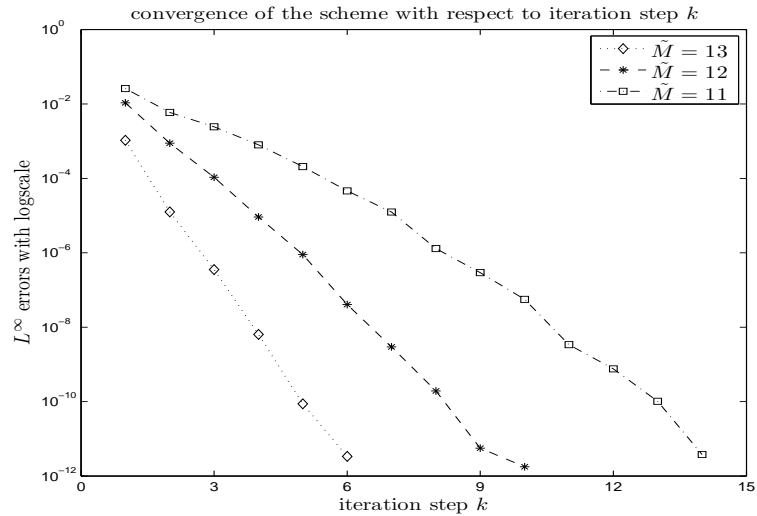}
\caption{$L^\infty $-errors versus iteration numbers $k$ with ${M}=25, N=20$,
 $\tilde{M}=11,12,13$.}
\label{fig2} 
\end{figure}

\begin{figure}[tb]
\centering
\includegraphics[width=11.5cm,height=6.9cm]{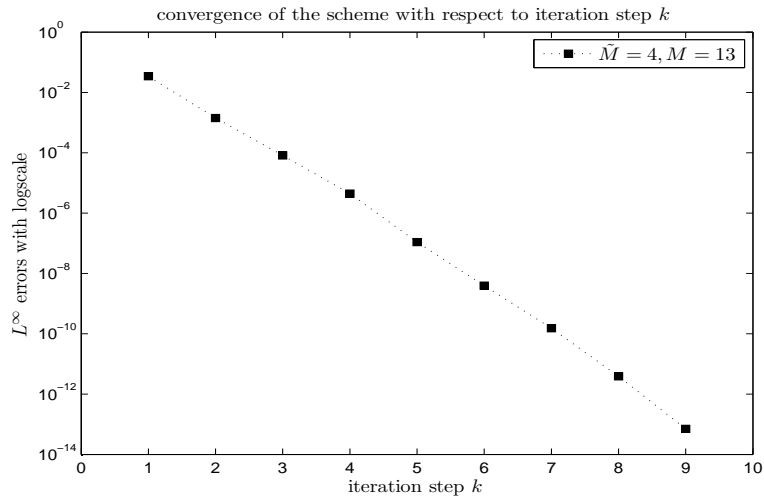}
\caption{$L^\infty $-errors versus iteration number $k$ with $M=13, N=20, \tilde M=4$, and $T=20$.}
\label{fig2b} 
\end{figure}

\begin{figure}[!htb]
\centering
\includegraphics[width=12cm,height=6.9cm]{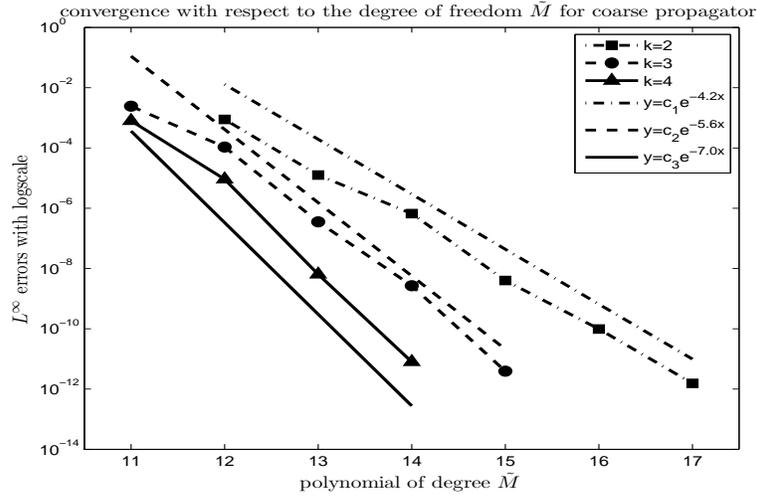}
\caption{$L^\infty $-errors versus the degree of freedom in coarse
approximation $\tilde{M}$ with
${M}=25$,$N=20$ and $k=2,3$, and 4.}
\label{fig3} 
\end{figure}

\newpage
\begin{example}
Consider the Volterra integral equation with the kernel of exponential form 
$K(t,s)=(t-s)e^{s-t}$ and exact solution $u(t) = \frac{1}{4}(2t-1+e^{-2t})$.
\end{example}

We repeat the investigation of the convergence behavior of
numerical solutions as in the first test. 
The $L^\infty$-errors of numerical solutions versus $M$ with $\tilde M=5, N=20$ 
are plotted in Figure \ref{fig5}. It is seen that the convergence remains exponential
with respect to $M$ if the scheme is iterated
a large enough step,
and it appears that the growth in $t$ of the exact solution does not affect 
the convergence rate even for $T = 100$.

%
%

\begin{figure}[tb]
\centering
\includegraphics[width=11cm,height=7cm]{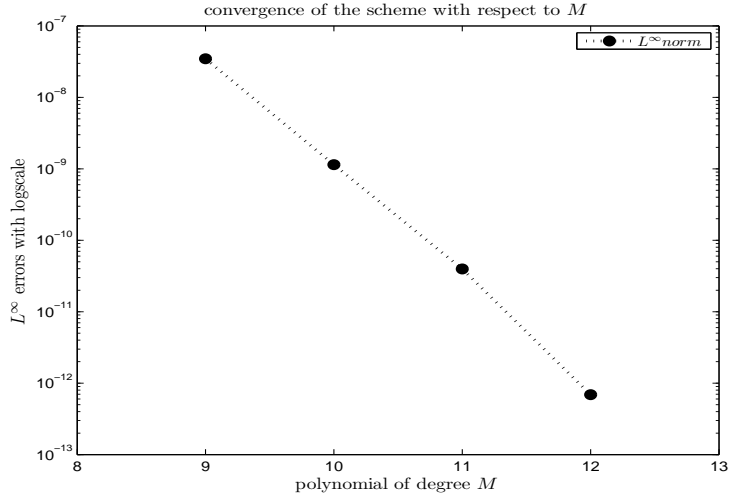}
\caption{$L^\infty $-errors versus $M$ with $\tilde{M}=5, N=20$.}
\label{fig5} 
\end{figure}

\section{Parallelism Efficiency}

Although not yet tested in a parallel machine, the parallelism efficiency of the proposed scheme can be 
investigated through a cost estimate. To simplify the cost estimation, we suppose that 
the inter-processor communication cost in the implementation of the parallel in time scheme is 
negligible as compared to the overall cost. The parallelism efficiency is demonstrated by a cost comparison
between the parallel in time scheme \eqref{n2e5}-\eqref{n2e5a} and the classical sequential scheme based on the 
corresponding fine mesh.
  
Firstly,  the classical sequential scheme based on the fine mesh consists of solving
the problems $\F_{n} \left(t,U_{1}, U_{2}, \cdots, U_{n-1}\right)$ 
consecutively for  $n=1,\cdots,N$.
The computational complexity is equal to the sum of
all the elementary operations in $I_n, n=1,\cdots,N$.
Denote the total computational cost by $\mathcal{C}_{\F}$ and the cost for 
$n$-th subproblem by $\mathcal{C}^n_{\F}$. Then
\[
\mathcal{C}_{\F}=\sum\limits_{n=1}^{N}\mathcal{C}^n_{\F}.
\]
Neglecting the cost of evaluating the integral terms on the right hand sides, 
the spectral discretization $\F_{n}$ produces an 
elementary cost $\mathcal{C}^n_\F $ 
approximately equal to
$\mathcal{O}(M^2 M)$,  where 
$\mathcal{O}(M^2)$ is the number of operations needed for
the matrix vector multiplication and $\mathcal{O}(M)$ 
is the estimated iteration number 
required to achieve the convergence of the iterative method.
As a result, the total computational complexity of the sequential fine solutions is
\bq\label{cost-1}
\mathcal{C}_{\F}=\mathcal{O}(NM^3).
\eq

If we implement the scheme \eqref{n2e5}-\eqref{n2e5a}  in a parallel architecture with enough processors, then
the total computational time corresponds to the cost to solve a sequential set of $N$ coarse subproblems
and a fine subproblem in a single processor.
The cost of solving the sequential set of $N$ coarse subproblems is estimated to be 
$\sum^N_{n=1} \mathcal{C}^n_{\G}$, 
where $\mathcal{C}^n_{\G}=\mathcal{O}(\tilde{M}^2+M\tilde{M})$ is the cost to
solve the $n$-th coarse subproblem, with $M\tilde M$ being
the cost of the fine-to-coarse interpolation. Note that here we only count 
the number of operations needed for
the matrix vector multiplication in the coarse mesh, which is equal to $\mathcal{O}(\tilde{M}^2)$,
because the Gauss-Seidel iterative method has been employed to solve the final linear system
with the previous solution as the initial guess, 
and it is found that the convergence was achieved within a few iterations.
For a same reason, the cost of solving a single fine subproblem is approximately 
$\mathcal{O}(M^2)$. Note also, in the implementation of the parallel in time scheme 
there is a need to interpolate the solution between the fine mesh and coarse mesh, the cost
of which is $\mathcal{O}(NM\tilde{M})$.  Therefore if $K$ is 
the number of iterations required to achieve the desired
convergence of the parallel in time algorithm, then the total computational complexity is
\br\label{cost-2}
K\Big[ \mathcal{O}(N\tilde{M}^2+NM\tilde{M}) + 
\mathcal{O}(M^2) +  \mathcal{O}(NM\tilde{M})\Big].
\er
Comparing \eqref{cost-1} with \eqref{cost-2}, we obtain a speed up 
(i.e., the percentage with respect the sequential scheme) close to
\bex
\mathcal{O}
\left(K\frac{\tilde{M}^2}{M^3}+K\frac{\tilde{M}}{M^2} +
\frac{K}{NM}\right)^{-1}.
\eex
It can be verified that in general case
the speed up is better than $\mathcal{O} (M/K)$.
In some particular cases, 
for example if the number of degrees of freedom for the coarse 
solver is far less than that of the fine solver, 
i.e.,
$\tilde{M} \ll M$ or $\tilde{M}N = M$, 
then the speed up by using the parallel in time algorithm would be 
close to
$\mathcal{O}(\frac{NM}{K})$.

\bibliographystyle{plain}

\end{document}